\documentclass[12pt,twoside]{amsart}
\usepackage{amsmath}
\usepackage{amsthm}
\usepackage{amsfonts}
\usepackage{amssymb}
\usepackage{latexsym}
\usepackage{mathrsfs}
\usepackage{amsmath}
\usepackage{amsthm}
\usepackage{amsfonts}
\usepackage{amssymb}
\usepackage{latexsym}
\usepackage{geometry}
\usepackage{dsfont}
\usepackage[dvips]{graphicx}
\usepackage{color}
\usepackage[all]{xy}

\date{}
\allowdisplaybreaks[4] \footskip=15pt
\renewcommand{\uppercasenonmath}[1]{}

\usepackage{graphicx,amssymb}
\usepackage[all]{xy}
\usepackage{amsmath}

\allowdisplaybreaks
\usepackage{amsthm}
\usepackage{color}

\theoremstyle{plain}
\newtheorem{theorem}{Theorem}[section]
\newtheorem{proposition}[theorem]{Proposition}
\newtheorem{lemma}[theorem]{Lemma}
\newtheorem{corollary}[theorem]{Corollary}
\newtheorem{example}[theorem]{Example}
\newtheorem*{open question}{Open Question}
\newtheorem{definition}[theorem]{Definition}

\theoremstyle{definition}
\newtheorem*{acknowledgement}{Acknowledgement}

\theoremstyle{remark}
\newtheorem{remark}[theorem]{Remark}

\newcommand{\Id}{\mathrm{Id}}

\def\p{\frak p}
\def\m{\frak m}

\def\GV{{\rm GV}}
\def\tor{{\rm tor_{\rm GV}}}
\def\Hom{{\rm Hom}}
\def\Ext{{\rm Ext}}

\def\Ker{{\rm Ker}}

\def\Im{{\rm Im}}
\def\Coker{{\rm Coker}}

\def\GV{{\rm GV}}
\def\Max{{\rm Max}}
\def\DW{{\rm DW}}

\def\Spec{{\rm Spec}}
\def\Max{{\rm Max}}
\begin{document}
\begin{center}
{\large  \bf On (uniformly) $S$-$w$-Noetherian rings and modules}

\vspace{0.5cm}  Xiaolei Zhang$^{a}$

{\footnotesize a.\ School of Mathematics and Statistics, Shandong University of Technology, Zibo 255049, China\\

}
\end{center}

\bigskip
\centerline { \bf  Abstract}
\bigskip
\leftskip10truemm \rightskip10truemm \noindent

Let $R$ be a commutative ring and $S$  a multiplicative subset of $R$. We introduce and study the notions of ($u$-)$S$-$w$-Noetherian modules and ($u$-)$S$-$w$-principal ideal modules. Some characterizations of these new concepts are given. 
\vbox to 0.3cm{}

\leftskip0truemm \rightskip0truemm
\bigskip

\section{Introduction}
Throughout this paper, $R$ is always a commutative ring with identity. The spectrum of all prime ideals and maximal ideals is denoted by $\Spec(R)$ and $\Max(R)$, respectively. For a subset $U$ of  an $R$-module $M$, we denote by $\langle U\rangle$ the $R$-submodule of $M$ generated by $U$. A subset $S$ of $R$ is called a multiplicative subset of $R$ if $1\in S$ and $s_1s_2\in S$ for any $s_1\in S$, $s_2\in S$. 

Recall from  Anderson and Dumitrescu \cite{ad02} that an ideal $I$ of a ring $R$ is said to be  \emph{$S$-finite} $($resp.,  \emph{$S$-principal}$)$ $($with respect to $s)$ if there is a finitely (resp., principally) generated sub-ideal $K$ of $I$  such that $sI\subseteq K$ for some $s\in S$. A ring $R$ is  called an \emph{$S$-Noetherian ring} (resp., \emph{$S$-PIR}) if every ideal of $R$ is $S$-finite (resp., $S$-principal). Futhermore, Cohen's Theorem, Eakin-Nagata Theorem and Hilbert Basis Theorem for $S$-Noetherian rings are obtained in \cite{ad02}. Since then, many algebraists  have paid considerable attention to the notion of $S$-Noetherian rings \cite{bh18,hh15,as16,l15,lO14,lO15,l07,kmz21}. Recently,  Qi and Kim et al. \cite{qkwcz22} introduced the notion of  \emph{uniformly $S$-Noetherian rings} (resp., \emph{uniformly $S$-PIRs}), i.e., rings $R$ satisfying that every ideal of $R$ is $S$-finite (resp., $S$-principal) with respect to a  ``uniform'' $s\in S$. Moreover, some moduletic characterizations of uniformly $S$-Noetherian rings are given in \cite{qkwcz22}.

Since the birth of star operation, a lot of attention has been paid to the ideal theory. In 1997, Wang and McCasland {\cite{fm97}} introduced the notion of $w$-operations over integral domains to study SM domains (i.e. $w$-Noetherian domains). The notion of $w$-modules they introduced makes it possible to study modules in terms of star operations. A significant development about $w$-theory is in the literature {\cite{hfxc11}} in 2011. The authors Yin et al. \cite{hfxc11}  extended $w$-operation theory to commutative rings with zero divisors to study the so-called $w$-Noetherian rings.  In the past decade, some works on $w$-Noetherian rings have continuously appeared \cite{x22,xw15,yc12,zwk,zkq23,zkzx22}. 

In 2014, Kim et al. \cite{kkl14} investigated an $S$-version of strong Mori domains, i.e., $S$-strong Mori domains. They actually gave the Hilbert basis Theorem and Cohen's Theorem for $S$-Mori domains. Subsequently,  Ahmed \cite{A18} studied $S$-Mori domains, which can be seen as the $S$-Noetherian domains under $v$-operation.

We give a quick review of star operations over integral domains and $w$-operations over commutative rings for convenience. For more details, the reader can consult {\cite[Chapters 6,7]{fk16}}.

Let $D$ be an integral domain with its quotient field $Q$ and $\mathscr{F}(D)$ the set of all nonzero
fractional ideals of $D.$ A set map $\ast:\mathscr{F}(D)\rightarrow \mathscr{F}(D)$ is said to be a  \emph{star operation} provided
it satisfies the following properties: for all $0\not=a \in Q$ and $A, B\in \mathscr{F}(D)$,  we have
\begin{enumerate}
	\item $\langle a\rangle_\ast=\langle a\rangle$ and $aA_{\ast}=(aA)_{\ast}$;
	\item $A\subseteq A_{\ast}$, and if $A\subseteq B$ then $A_{\ast}\subseteq B_{\ast}$;
	\item $A_{\ast}= (A_{\ast})_{\ast}$.
\end{enumerate}
For example, $v$-operation, $t$-operation and $w$-operation are three classical star operations over integral domains.

In 2011, Yin et al. \cite{hfxc11}  extended $w$-operation theory to commutative rings. Let $R$ be a commutative ring and $J$ a finitely generated ideal of $R$. Then $J$ is called a \emph{$\GV$-ideal} if the natural homomorphism $R\rightarrow \Hom_R(J,R)$ is an isomorphism. The set of all $\GV$-ideals is denoted by $\GV(R)$. Let $M$ be an $R$-module. Define

\begin{center}
	{\rm $\tor(M):=\{x\in M|Jx=0$, for some $J\in \GV(R) \}.$}
\end{center}
An $R$-module $M$ is said to be \emph{$\GV$-torsion} (resp. \emph{$\GV$-torsion-free}) if $\tor(M)=M$ (resp. $\tor(M)=0$). A $\GV$-torsion-free module $M$ is called a \emph{$w$-module} if $\Ext_R^1(R/J,M)=0$ for any $J\in \GV(R)$, and the \emph{$w$-envelope} of $M$ is given by
\begin{center}
	{\rm $M_w:=\{x\in E(M)|Jx\subseteq M$, for some $J\in \GV(R) \},$}
\end{center}
where $E(M)$ is the injective envelope of $M$. Therefore, $M$ is a $w$-module if and only if $M_w=M$.  A \emph{$\DW$ ring} $R$ is a ring for which every $R$-module is a $w$-module. 
The set of all prime $w$-ideals is denoted by $ w$-$\Spec(R)$.
A \emph{maximal $w$-ideal} is an ideal of $R$ which is maximal among the $w$-submodules of $R$. The set of all maximal $w$-ideals is denoted by $ w$-$\Max(R)$. All maximal $w$-ideals are  prime ideals (see {\cite[Theorem 6.2.14]{fk16}}).

 A sequence $A\rightarrow B\rightarrow C$ of $R$-modules is said to be \emph{$w$-exact} if for any $\p\in w$-$\Max(R)$, $A_\p\rightarrow B_\p\rightarrow C_\p$ is exact. An $R$-homomorphism $f:M\rightarrow N$ is said to be a \emph{$w$-monomorphism} (respectively, \emph{$w$-epimorphism}, \emph{$w$-isomorphism}) if for any $ p\in w$-$\Max(R)$, $f_\p:M_\p\rightarrow N_\p$ is a monomorphism (respectively, an epimorphism, an isomorphism). Note that $f$ is a $w$-monomorphism (resp. $w$-epimorphism) if and only if $\Ker(f)$ (resp. $\Coker(f)$) is $\GV$-torsion. 
A class $\mathcal{C}$ of $R$-modules is said to be \emph{closed under $w$-isomorphisms} provided that for any $w$-isomorphism $f:M\rightarrow N$, if one of the modules $M$ and $N$ is in $\mathcal{C}$, so is the other.
One can show that a class $\mathcal{C}$ of $R$-modules is closed under $w$-isomorphisms if and only if any $R$-module $N$ which is $w$-isomorphic to a module in $\mathcal{C}$ is also in $\mathcal{C}$.

Let $M$ be an $R$-module. Then $M$ is said to be \emph{$w$-finite} (resp., \emph{$w$-principal}) if there exist a finitely generated (resp., rank one) free module $F$ and a $w$-epimorphism $g: F\rightarrow M$. The classes of $w$-finite modules are all closed under $w$-isomorphisms (\cite[Corollary 6.4.4]{fk16}). $M$ is called a \emph{$w$-Noetherian module} (resp., \emph{$w$-PIM})
if every submodule of $M$ is  $w$-finite (resp, \emph{$w$-principal}). A ring $R$ is called a \emph{$w$-Noetherian ring} (resp., $w$-PIR) if
$R$ as an $R$-module is a $w$-Noetherian module (resp., $w$-PIM).  The classes of $w$-Noetherian modules and $w$-PIMs  are all closed under $w$-isomorphisms(\cite[Corollary 6.8.3]{fk16}).

\section{basic properties of $(u$-$)S$-$w$-Noetherian modules and $(u$-$)S$-PIMs}
We begin with the notions of $S$-$w$-finite modules and $S$-$w$-principal modules.

\begin{definition}
	
Let $R$ be a ring, $S$ a multiplicative subset of $R$ and $M$ an $R$-module. Then $M$ is said to be $S$-$w$-finite $($resp., $S$-$w$-principal$)$ $($with respect to $s)$, if there exists an $s\in S$ and a finitely  $($resp., principally$)$ generated submodule $N$ of $M$ such that $s(M/N)$ is $\GV$-torsion.
\end{definition}

\begin{remark} The following statements hold.
\begin{enumerate}
\item  $S$-finite (resp., $S$-principal) modules and $w$-finite (resp., $w$-principal) modules are all $S$-$w$-finite (resp., $S$-$w$-principal).
\item A $\GV$-torsion-free module $M$ is $S$-$w$-finite ($S$-$w$-principal) (with respect to $s$) if and only if $sM\subseteq N_w$ for some finitely  (resp., principally) generated submodule $N$ of $M$.
\item  The classes of $S$-$w$-finite modules and  $S$-$w$-principal modules (with respect to $s$) are all closed under $w$-isomorphisms (see \cite[Corollary 6.4.4]{fk16} for classical case).
\end{enumerate}
\end{remark}

Recall from \cite{ad02} that an $R$-moduel $M$ is said to be an \emph{$S$-Noetherian module} (resp., \emph{$S$-principally ideal module} (or \emph{$S$-PIM})) if each submodule  of $M$ is $S$-finite (resp., $S$-principal); $M$ is said to be a \emph{$u$-$S$-Noetherian module} (resp., \emph{$u$-$S$-principally ideal module} (or \emph{$u$-$S$-PIM})) if there is $s\in S$ such that  each submodule  of $M$ is $S$-finite (resp., $S$-principal) with respect to $s\in S$. Now we introduce some generalizations of there concepts in terms of  $w$-operations.

\begin{definition} Let $R$ be a ring, $S$ a multiplicative subset of $R$ and $M$ an $R$-module. Then
\begin{enumerate}
\item  An $R$-module $M$ is said to be an $S$-$w$-Noetherian module $($resp., $S$-$w$-principally ideal module $($or $S$-$w$-PIM$))$ if each submodule  of $M$ is $S$-$w$-finite $($resp., $S$-$w$-principal$).$ 
\item $R$ is called an $S$-$w$-Noetherian ring $($resp., $S$-$w$-PIR$)$ if $R$ itself is an  $S$-$w$-Noetherian module $($resp., $S$-$w$-PIM$).$
\item 	 $M$ is called a uniformly $S$-$w$-Noetherian $($or $u$-$S$-$w$-Noetherian$)$ module $($resp., $u$-$S$-$w$-principally ideal module $($or $u$-$S$-$w$-PIM$))$ $($with respect to $s)$ provided there exists an element $s\in S$ such that each submodule of $M$ is $S$-$w$-finite $($resp., $S$-$w$-principal$)$ with respect to $s$. 
\item  $R$ is called $u$-$S$-$w$-Noetherian ring $($resp., $u$-$S$-$w$-PIR$)$  $($with respect to $s)$ if $R$ itself is a  $u$-$S$-$w$-Noetherian module $($resp., $u$-$S$-$w$-PIM$)$ $($with respect to $s)$.
	 	\end{enumerate}
\end{definition}

Clearly, all modules are $u$-$S$-$w$-Noetherian modules $($resp., $u$-$S$-$w$-PIMs$)$ if $0\in S$;  and ($u$-)$S$-$w$-Noetherian modules $($resp., ($u$-)$S$-$w$-PIMs$)$ are $w$-Noetherian modules $($resp., $w$-PIMs$)$ if $S$ is consist of units.

Let $S$ be a multiplicative subset of $R$. The saturation $S^{\ast}$ of $S$ is defined as $S^{\ast}=\{ s\in R \mid s_1=ss_2$ for some $s_1,s_2\in S\}$. A multiplicative subset  $S$ of $R$ is called \emph{ saturated} if $S=S^{\ast}$. Note that $S^{\ast}$ is always a saturated multiplicative subset containing $S$.

\begin{proposition}\label{consar}
	Let $R$ be a ring. Then the following statements hold.
	\begin{enumerate}
		\item Let $S \subseteq T$ be multiplicative subsets of $R$ and $M$ an $R$-module. If $M$ is a $(u$-$)S$-$w$-Noetherian module $($resp., $(u$-$)S$-$w$-PIM$),$ then $M$ is a $(u$-$)T$-$w$-Noetherian module $($resp., $(u$-$)T$-$w$-PIM$).$
		\item $M$ is  a $(u$-$)S$-$w$-Noetherian module $($resp., $(u$-$)S$-$w$-PIM$)$ if and only if $R$ is a $(u$-$)S^{\ast}$-$w$-Noetherian module $($resp., $(u$-$)S^{\ast}$-$w$-PIM$),$ where $S^{\ast}$ is the saturation of $S$.
	\end{enumerate}
\end{proposition}
\begin{proof} We only prove them for the uniform $S$-version since the $S$-version can be proved similarly.
	
$(1)$: Obvious.
	
$(2)$: Suppose $M$ is a $u$-$S^{\ast}$-$w$-Noetherian module. Then there is a $t\in S^{\ast}$ such that for any submodule $N$ of $M$, there is a finitely generated (resp., principally generated) submodule $K$ of $N$ satisfying $t(N/K)$ is $\GV$-torsion. Assume $s=ts_1$ with $s,s_1\in S$. Then $s(N/K)=ts_1(N/K)\subseteq t(N/K)$ is $\GV$-torsion. So $N$ is $w$-finite (resp., $w$-principal) with respect to $s$. Hence $M$ is a $u$-$S$-$w$-Noetherian module. 
\end{proof}

Recall from \cite{z21}, an $R$-sequence  $M\xrightarrow{f} N\xrightarrow{g} L$ is called  \emph{$u$-$S$-exact} provided that there is an element $s\in S$ such that $s\Ker(g)\subseteq \Im(f)$ and $s\Im(f)\subseteq \Ker(g)$. An $R$-homomorphism $f:M\rightarrow N$ is an \emph{$u$-$S$-monomorphism}  $($resp.,   \emph{$u$-$S$-epimorphism}, \emph{$u$-$S$-isomorphism}$)$ provided $0\rightarrow M\xrightarrow{f} N$   $($resp., $M\xrightarrow{f} N\rightarrow 0$, $0\rightarrow M\xrightarrow{f} N\rightarrow 0$ $)$ is   $u$-$S$-exact. It is easy to verify that an  $R$-homomorphism $f:M\rightarrow N$ is a $u$-$S$-monomorphism $($resp., $u$-$S$-epimorphism$)$ if and only if  $\Ker(f)$ $($resp., $\Coker(f))$ is a  $u$-$S$-torsion module. Different with $w$-isomorphisms, every $u$-$S$-isomorphism has an ``inverse'' mapping, i.e., an $R$-homomorphism $f:M\rightarrow N$ is a $u$-$S$-isomorphism if and only if there is a $u$-$S$-isomorphism $g:N\rightarrow M$ such that $fg=s\Id=gf$ (\cite[Proposition 1.1]{z23}).

Next, we will investigate the transfer of $u$-$S$-$w$-Noetherian properties under the two kinds of generalized short exact sequences.

\begin{lemma}\label{s-exct-tor}\cite[Proposition 2.8]{z21}
	Let $R$ be a ring and $S$ a multiplicative subset of $R$.  Let $0\rightarrow A\xrightarrow{f} B\xrightarrow{g} C\rightarrow 0$  be a $u$-$S$-exact sequence. Then  $B$ is  $u$-$S$-torsion if and only if $A$ and $C$ are $u$-$S$-torsion.
\end{lemma}

\begin{lemma}\label{s-exct-diag}\cite[Lemma 2.11]{qkwcz22}
	Let $R$ be a ring and $S$ a multiplicative subset of $R$.  Let
	$$\xymatrix@R=20pt@C=25pt{
		0 \ar[r]^{}&A_1\ar@{^{(}->}[d]^{i_A}\ar[r]& B_1 \ar[r]^{\pi_1}\ar@{^{(}->}[d]^{i_B}&C_1\ar[r] \ar@{^{(}->}[d]^{i_C} &0\\
		0 \ar[r]^{}&A_2\ar[r]&B_2 \ar[r]^{\pi_2}&C_2\ar[r] &0\\}$$
	be a commutative diagram with exact rows, where $i_A, i_B$ and $i_C$ are embedding maps. Suppose $s_A A_2\subseteq A_1$ and $s_C C_2\subseteq C_1$ for some $s_A\in S, s_C\in S$. Then $s_As_CB_2\subseteq B_1$.
\end{lemma}

\begin{proposition} \label{s-u-noe-exact}
	Let $R$ be a ring and $S$ a multiplicative subset of $R$. Let $0\rightarrow A\rightarrow B\rightarrow C\rightarrow 0$ be an exact sequence of $w$-modules. Then $B$ is $(u$-$)S$-$w$-Noetherian if and only if $A$ and $C$ are $(u$-$)S$-$w$-Noetherian.
\end{proposition}
\begin{proof} We only prove them for the uniform $S$-version since the $S$-version can be proved similarly.
	
It is easy to verify that if $B$ is $u$-$S$-$w$-Noetherian, so are $A$ and $C$. Suppose $A$ and $C$ are $u$-$S$-$w$-Noetherian. Let $\{B_i\}_{i\in \Lambda}$ be the set of all $w$-submodules of $B$. Since $A$ is $u$-$S$-$w$-Noetherian, there exists an element $s_1\in S$ such that  $s_1(A\cap B_i)\subseteq (K_i)_w\subseteq A\cap B_i$ for some finitely generated $R$-module $K_i$ and any $i\in \Lambda$.  Since $C$ is $u$-$S$-$w$-Noetherian, there also exists an element $s_2\in S$ such that $s_2((B_i+A)/A)\subseteq (L_i)_w\subseteq ((B_i+A)/A)_w$ for some finitely generated $R$-module $L_i$ and any $i\in \Lambda$ . Let $N_i$ be the finitely generated submodule of $B_i$ generated by the finite generators of $K_i$ and finite pre-images of generators of  $L_i$.  Consider the following natural commutative diagram with exact rows: $$\xymatrix@R=20pt@C=25pt{
		0 \ar[r]^{}&K_i\ar@{^{(}->}[d]\ar[r]&N_i \ar[r]\ar@{^{(}->}[d]&L_i\ar[r] \ar@{^{(}->}[d] &0\\
		0 \ar[r]^{}&A\cap B_i\ar[r]&B_i \ar[r]&(B_i+A)/A \ar[r] &0.\\}$$
	Set $s=s_1s_2\in S$. We have  $sB_i\subseteq N_i\subseteq B_i$ by Lemma \ref{s-exct-diag}. Hence 
	$sB_i\subseteq (N_i)_w\subseteq B_i$. It follows that  $B$ is $u$-$S$-$w$-Noetherian with respect to $s$.
\end{proof}

\begin{corollary} \label{s-u-noe-exact-m}
	Let $R$ be a ring and $S$ a multiplicative subset of $R$. If $R$ is a $(u$-$)S$-$w$-Noetherian ring, then any finitely generated $R$-module is  $(u$-$)S$-$w$-Noetherian.
\end{corollary}
\begin{proof} Suppose $M$ is a finitely generated $R$-module. Then there is an exact sequence $0\rightarrow K\rightarrow F\rightarrow M\rightarrow 0$ where $F$ is finitely generated free $R$-module. Using the induction on the rank of $F$, one can show $F$ is $(u$-$)S$-$w$-Noetherian by Proposition \ref{s-u-noe-exact}. So $M$ is also $(u$-$)S$-$w$-Noetherian by Proposition \ref{s-u-noe-exact} again.
\end{proof}

\begin{proposition} \label{s-u-noe-s-exact}
	Let $R$ be a ring and $S$ a multiplicative subset of $R$. Let $0\rightarrow A\rightarrow B\rightarrow C\rightarrow 0$ be a $u$-$S$-exact sequence of $w$-modules. Then $B$ is $(u$-$)S$-$w$-Noetherian if and only if $A$ and $C$ are $(u$-$)S$-$w$-Noetherian
\end{proposition}
\begin{proof}
	We only prove them for the uniform $S$-version since the $S$-version can be proved similarly.
	
  Let $0\rightarrow A\xrightarrow{f} B\xrightarrow{g} C\rightarrow 0$ be a $u$-$S$-exact sequence. Then there exists an element $s\in S$ such that  $ s\Ker(g)\subseteq  \Im(f)$ and $ s\Im(f)\subseteq  \Ker(g)$. Note that $\Im(f)/s\Ker(g)$ and $\Ker(g)/s\Im(f)$ are $u$-$S$-torsion. If $\Im(f)$ is $u$-$S$-$w$-Noetherian, then the submodule $s\Im(f)$ of $\Im(f)$ is $u$-$S$-$w$-Noetherian. Thus $\Ker(g)$ is $u$-$S$-$w$-Noetherian by Proposition \ref{s-u-noe-exact}. Similarly, if  $\Ker(g)$ is $u$-$S$-$w$-Noetherian, then $\Im(f)$ is $u$-$S$-Noetherian. Consider the following three exact sequences:
	$0\rightarrow\Ker(g) \rightarrow  B\rightarrow \Im(g)\rightarrow 0,\quad 0\rightarrow\Im(g) \rightarrow  C\rightarrow \Coker(g)\rightarrow 0,$ and $0\rightarrow\Ker(f) \rightarrow  A\rightarrow \Im(f)\rightarrow 0$
	with $\Ker(f)$ and $\Coker(g)$ $u$-$S$-torsion.   It is easy to verify that $B$ is $u$-$S$-$w$-Noetherian if and only if $A$ and $C$ are $u$-$S$-$w$-Noetherian by Proposition \ref{s-u-noe-exact}.
\end{proof}

\begin{corollary} \label{s-u-noe-u-iso}
	Let $R$ be a ring, $S$ a multiplicative subset of $R$ and $ M\xrightarrow{f} N$  a $u$-$S$-isomorphism. If one of $M$ and $N$ is $(u$-$)S$-$w$-Noetherian, so is the other.
\end{corollary}
\begin{proof} It follows from Proposition \ref{s-u-noe-s-exact} since $0\rightarrow M\xrightarrow{f} N\rightarrow 0\rightarrow 0$ is a $u$-$S$-exact sequence.
\end{proof}

\begin{proposition} \label{s-w-noe-s-exact}
	Let $R$ be a ring and $S$ a multiplicative subset of $R$. Let $0\rightarrow A\rightarrow B\rightarrow C\rightarrow 0$ be a $w$-exact sequence of $R$-modules. Then $B$ is $(u$-$)S$-$w$-Noetherian if and only if $A$ and $C$ are $(u$-$)S$-$w$-Noetherian
\end{proposition}
\begin{proof}
	The proof is similar with that of Proposition \ref{s-u-noe-s-exact} and \cite[Theorem 6.8.2]{fk16}.	
\end{proof}

Similarly, we have the following corollaries:
\begin{corollary} 
Let $R$ be a ring and $S$ a multiplicative subset of $R$. Let $0\rightarrow A\rightarrow B\rightarrow C\rightarrow 0$ be a $u$-$S$-exact sequence of $R$-modules. Then $B$ is $(u$-$)S$-$w$-Noetherian if and only if $A$ and $C$ are $(u$-$)S$-$w$-Noetherian
\end{corollary}
\begin{corollary} 
	Let $R$ be a ring, $S$ a multiplicative subset of $R$ and $ M\xrightarrow{f} N$  a $w$-isomorphism. If one of $M$ and $N$ is $(u$-$)S$-$w$-Noetherian, so is the other.
\end{corollary}
\begin{corollary} \label{s-w-noe-u-iso}
	Let $R$ be a ring, $S$ a multiplicative subset of $R$. The direct sum $\bigoplus\limits_{i=1}^n M_i$  is $(u$-$)S$-$w$-Noetherian if and only if so is each $M_i$.
\end{corollary}
\begin{proof}
	It follows by Proposition \ref{s-w-noe-s-exact}.
\end{proof}

The rest of this section is devoted to give some differences and relations of classical and new concepts. 
In general, we have the following implications:

$$
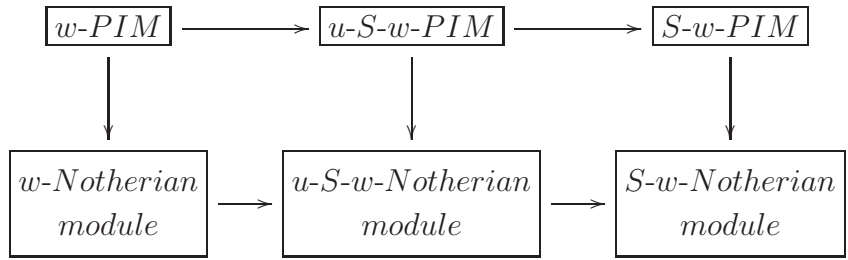
\xymatrix@R=30pt@C=20pt{
	\boxed{w\mbox{-}PIM} \ar[d]_{}\ar[r]^{} & \boxed{u\mbox{-}S \mbox{-}w\mbox{-}PIM}\ar[r]^{}\ar[d]_{} &\boxed{S\mbox{-}w\mbox{-}PIM}  \ar[d]_{} \\
\boxed{\begin{matrix}w\mbox{-}Notherian\\  module\end{matrix}} \ar[r]^{} &\boxed{\begin{matrix} u\mbox{-}S\mbox{-}w\mbox{-}Notherian\\ module\end{matrix}} \ar[r]^{} &\boxed{\begin{matrix}S\mbox{-}w\mbox{-}Notherian\\ module\end{matrix}}\\ }$$


\begin{proposition}\label{mutl-usm}
	Let $R_i$ be a ring, $S_i$  a multiplicative subset of $R_i$ and $M_i$ an $R_i$-module $(i=1,2)$. Set
	$R= R_1\times R_2, S=S_1\times S_2$ and $M=M_1\times M_2$. Then $M$ is a $(u$-$)S$-$w$-Noetherian module $($resp., $(u$-$)S$-$w$-PIM$)$ over $R$ if and only if each  $M_i$ is a $(u$-$)S_i$-$w$-Noetherian module $($resp., $(u$-$)S_i$-$w$-PIM$)$ over $R_i$.
\end{proposition}
\begin{proof} This result easily follows by that an ideal $I=I_1\times I_2\in \GV(R)$ if and only if each $I_i\in \GV(R_i)$. So we omit its proof.
\end{proof}

The following example shows that $u$-$S$-$w$-Noetherian modules  $($resp., $u$-$S$-$w$-PIMs$)$ need not be $w$-Noetherian modules $($resp., $w$-PIMs$)$ in general.
\begin{example} Let $M_1$ be a $w$-Noetherian module $($resp., $w$-PIM$)$ over a ring $R_1$, and $M_2$ be a non-$w$-Noetherian module $($resp., non-$w$-PIM$)$ over a ring $R_2$. Set $R=R_1\times R_2$, $M=M_1\times M_2$ and $S=\{1\}\times\{0,1\}$. Then $M$ is a $u$-$S$-$w$-Noetherian module $($resp., $u$-$S$-$w$-PIM$)$  but not a $w$-Noetherian module $($resp., $w$-PIM$)$ by Proposition \ref{mutl-usm}.
\end{example}

A multiplicative subset $S$ of $R$ is said to satisfy the \emph{maximal multiple condition} if there
exists an $s\in S$ such that $t|s$ for each $t\in S$. Both finite multiplicative subsets and the  multiplicative subsets that consist of units  satisfy the maximal multiple condition.

\begin{proposition} \label{s-loc-u-noe-fini}
	Let $R$ be a ring, $S$ a multiplicative subset of $R$ satisfying maximal multiple condition and $M$ an $R$-module. Then $M$ is a $u$-$S$-$w$-Noetherian module $($resp., $u$-$S$-$w$-PIM$)$ if and only if $M$ is an  $S$-$w$-Noetherian module  $($resp.,  $S$-$w$-PIM$)$.
\end{proposition}
\begin{proof}  Suppose  $M$ is a $u$-$S$-$w$-Noetherian module $($resp.,$u$-$S$-$w$-PIM$)$. Then trivially $M$ is an $S$-$w$-Noetherian module (resp., $S$-$w$-PIM). Let $s\in S$ such that $t|s$ for each $t\in S$. Suppose $M$ is an $S$-$w$-Noetherian module (resp., $S$-$w$-PIM). Then for a submodule $N$ of $M$, there is a finitely (resp., principally) generated submodule $K$ of $N$
	such that $s_K(N/K)$ is $\GV$-torsion for some $s_N\in S$. Then $s(N/K)\subseteq s_N(N/K)$ is $\GV$-torsion. So $M$ is a $u$-$S$-$w$-Noetherian module $($resp., $u$-$S$-$w$-PIM$)$ with respect to $s$.
\end{proof}

Let $R$ be a ring, $M$ an $R$-module and $S$ a multiplicative subset of $R$. For any $s\in S$, there is a  multiplicative subset $S_s=\{1,s,s^2,....\}$ of $S$. We denote by $M_s$ the localization of $M$ at $S_s$. Certainly, $M_s\cong M\otimes_RR_s$. A multiplicative subset of $R$ is said to be regular if it consists of non-zero-divisors.

\begin{proposition} \label{s-loc-u-noe}
	Let $R$ be a ring, $S$ a regular multiplicative subset of $R$ and $M$ an $R$-module. If  $M$ is a $u$-$S$-$w$-Noetherian module $($resp., $u$-$S$-$w$-PIM$)$ over $R$, then there exists an element $s\in S$ such that $M_{s}$ is a $w$-Noetherian module $($resp., $w$-PIM$)$ over $R_s$.
\end{proposition}
\begin{proof} Since $M$ is a $u$-$S$-$w$-Noetherian module $($resp., $u$-$S$-$w$-PIM$)$, there exists an element $s\in S$ satisfies that for any submodule $N$ of $M$ there is a finitely (resp., principally) generated submodule $K$ of $N$ such that $s(N/K)$ is $\GV$-torsion.  So for any $n\in N$, there is $J\in\GV(R)$ such that  $Jsn\subseteq K$. Hence for any  $\frac{n}{s^m}\in N_s$ we have $J_s\frac{n}{s^m}\subseteq K_s$. Since $J_s\in\GV(R_s)$ by \cite[Lemma 4.3(1)]{WQ15}, we have $N_s/K_s$ is a $\GV$-torsion $R_s$-module. Consequently, $M_{s}$ is a $w$-Noetherian module (resp., $w$-PIM) over $R_s$.
\end{proof}

The following example shows that $S$-$w$-Noetherian rings  $($resp., $S$-$w$-PIRs$)$ need not be $u$-$S$-$w$-Noetherian rings  $($resp., $u$-$S$-$w$-PIRs$)$ in general.
\begin{example} Let $R=\mathbb{Q}+x\mathbb{R}[[x]]$ be the subring of formal power series ring $T=\mathbb{R}[[x]]$ with constants in $\mathbb{R}$ the set of all real numbers, where $\mathbb{Q}$ is the set of all rational numbers.  Indeed, let $0\not=s=a+xf(x)\in R$. We divide it into two cases. Case I: $a\not=0$. In this case, $s$ is a unit in $R$, and so  $R_{s}\cong R$. So $R$ can fit into a Milnor square of type II:
$$\xymatrix@R=32pt@C=35pt{R\ar@{^{(}->}[r]^{}\ar@{->>}[d]_{} & \mathbb{R}[[x]] \ar@{->>}[d]\\
\mathbb{Q} \ar@{^{(}->}[r]^{} & \mathbb{R}.}$$ Case II:  $a=0$. In this case, $R_s\cong \mathbb{Q}+(x\mathbb{R}[[x]])_{xf(x)}\cong \mathbb{Q}+(x\mathbb{R}[[x]])_{x}$. So $R_s$ can fit into a Milnor square of type II:
	$$\xymatrix@R=32pt@C=35pt{
		R_s\ar@{^{(}->}[r]^{}\ar@{->>}[d]_{} & \mathbb{R}[[x]][x^{-1}] \ar@{->>}[d]\\
		\mathbb{Q} \ar@{^{(}->}[r]^{} & \mathbb{R}.}$$
	Hence $R_s$ is not $w$-Noetherian by \cite[Theorem 8.4.16]{fk16} in  both cases. 
	
Set $S=R-\{0\}$. Then it is easy to verify that $R$ is an $S$-PIR, and so is an  $S$-$w$-PIR and is an  $S$-$w$-Noetherian ring. But $R$ is not a  $u$-$S$-$w$-Noetherian ring by Proposition \ref{s-loc-u-noe}. 
\end{example}

\section{On characterizations of $u$-$S$-$w$-Noetherian rings and modules}

It is well known that an $R$-module $M$ is a Noetherian module (resp., PIM) if and only if every countably generated submodule of $M$ is finitely (resp., principally) generated. Now we consider some more general versions. 
\begin{theorem}\label{u-countably}
	Let $R$ be a ring, $S$ an anti-Archimedean multiplicative subset of $R$ and $M$ an $R$-module. Then $M$ is a $u$-$S$-$w$-Noetherian module $($resp., $u$-$S$-$w$-PIM$)$ if and only if there is $s\in S$ such that  every countably generated submodule of $M$ is $S$-$w$-finite $($resp., $S$-$w$-principal$)$ with respect to $s$.
\end{theorem}
\begin{proof} The necessity clearly holds. Now we prove the sufficiency. 	Let $N=\langle f_{\alpha}\mid \alpha<\gamma\rangle$ be an $R$-submodule of $M$, where $\gamma$ is a cardinal. Consider the index set $$\Gamma=\{\beta\mid s(\langle f_{\alpha}\mid\alpha\leq\beta\rangle/\langle f_{\alpha}\mid\alpha<\beta\rangle)\ \mbox{is not GV-torsion}\}.$$
	
Let $t\in\bigcap\limits_{n\geq 1}s^nR\bigcap S$.	Then we claim that $t(N/ \langle f_\beta\mid \beta\in\Gamma\rangle)$ is $\GV$-torsion. Indeed, let  $f_{\beta'}\in N$. If $\beta'\in\Gamma$, we have done. Otherwise, $J_1sf_{\beta'}\subseteq \langle f_{\alpha}\mid\alpha<\beta'\rangle$  for some $J_1\in\GV(R)$. Suppose $\beta'$ is a limit ordinal, then  $J_1sf_{\beta'}\subseteq \langle f_{\alpha}\mid\alpha<\beta''\rangle$ for some successor ordinal $\beta''$, since $J_1sf_{\beta'}$ is finitely generated. So we only consider the successor ordinal case. Now assume $\beta'$ is a successor ordinal, then $J_1sf_{\beta'}\subseteq \langle f_{\alpha}\mid\alpha\leq \beta'-1\rangle$, and hence every element in $J_1sf_{\beta'}$ is a finite linear combination of $f_\alpha$ with $\alpha\leq \beta'-1$. If all such finite $\alpha$ are in $\Gamma$, we have done. Otherwise, we may assume each $f_\alpha$ with $\alpha\not\in\Gamma$ satisfy $J_2sf_\alpha$  is a finite linear combination of $f_\alpha$ with $\alpha\leq \beta'-2.$ Hence every element in  $J_1J_2s^2f_{\beta'}$ is a finite linear combination of $f_\alpha$ with $\alpha\leq \beta'-2.$ Such iterating steps terminate in finite steps. Indeed,  on contrary assume the iterating steps don't terminate. Denote by $A$ the set of all such $\alpha$. Then $A$ is countable. Denote by $K$ the $R$-module generated by all such  $f_\alpha$ with $\alpha\in A$. So there is a finitely (resp., principally) generated submodule $K'$ such that $s(K/K')$ is $\GV$-torsion by assumption. Taking $B$ to be the set of $\beta$ such that $f_{\beta}$ appears in the finite generators of $K'$, we have $B$ is finite and  $J_0sf_{\beta'}\subseteq\langle f_\beta\mid \beta\in B\rangle$. So each $\alpha\in A-B$ is not in $\Gamma$, and hence the iterating steps terminate in finite steps. It follows that there exists $J\in\GV(R)$ and an integer $n$ such that every element in $Js^nf_{\beta'}$ is a finite linear combination of $f_\alpha$ with $\alpha\in \Gamma,$ and hence $t(N/ \langle f_\beta\mid \beta\in\Gamma\rangle)$ is $\GV$-torsion.

If $\Gamma$ is finite, then $N$ is $S$-$w$-finite with respect to $t$, and so we have done (resp., by assumption). Otherwise assume $\Gamma$ is infinite. Take $\Gamma'=\{\beta_1<\beta_2<\cdots\}\subseteq \Gamma.$ Note that $\langle f_\beta\mid \beta\in\Gamma'\rangle$ is $S$-$w$-finite (resp., $S$-$w$-principal) with respect to $s$. Since every element in $\langle f_\beta\mid \beta\in\Gamma'\rangle$ is a finite combination of $f_\beta$, there is a finite subset $\Gamma''$	of $\Gamma'$ such that $s(\langle f_\beta\mid \beta\in\Gamma'\rangle/\langle f_\beta\mid \beta\in\Gamma''\rangle)$ is $\GV$-torsion. So there exists an integer $i$ such that no $\beta_j\in \Gamma'$ with $j\geq i$, which is a contradiction. Consequently, $M$ is a $u$-$S$-$w$-Noetherian module with respect to $t$.
\end{proof}

It was proved in \cite[Theorem 6.8.4]{fk16} that a $w$-module $M$ is a $w$-Noetherian module if and only if  every ascending chain of $w$-submodules is stable. Now, we give the Eakin-Nagata-Formanek Theorem for $u$-$S$-$w$-Noetherian modules.

\begin{theorem} \label{u-s-noe-char} {\bf (Eakin-Nagata-Formanek Theorem for $u$-$S$-$w$-Noetherian modules)}
	Let $R$ be a ring and $S$ an anti-Archimedean multiplicative subset of $R$. Let $M$ be a $w$-module.  Then the following statements are equivalent:
	\begin{enumerate}
		\item $M$ is $u$-$S$-$w$-Noetherian;
		\item there exists $s\in S$ such that any  ascending chain of $w$-submodules of $M$ is stationary with respect to $s$.
	\end{enumerate}
\end{theorem}
\begin{proof} $(1)\Rightarrow (2):$ Suppose $M$ is $u$-$S$-$w$-Noetherian with respect to some $s\in S$. Let $M_1\subseteq M_2\subseteq\cdots$  be an  ascending chain of $w$-submodules of $M$. Set $M_0=\bigcup\limits_{i=1}^{\infty}M_i$. Then $M_0$ is a $w$-module, and so there exists a finitely generated submodule $N_i$ of $M_i$ such that $sM_i\subseteq (N_i)_w$ for each $i\geq 0$. 
	Since $N_0$ is finitely generated, there exists $k\geq 1$ such that $N_0\subseteq M_k$. It follows that  $sM_0\subseteq (N_0)_w\subseteq M_k$. So $sM_n\subseteq M_k$ for any $n\geq k$.
	
$(2)\Rightarrow (1):$ Let $M$ be a countably generated $R$-module, say generated by $\{m_1,m_2,\dots\}.$ Set $M_n=\langle m_1,\dots,m_n\rangle$ for each $n\geq 1$. Then $(M_1)_w\subseteq (M_2)_w\subseteq\cdots$ is an  ascending chain of $w$-submodules of $M_w=M$. So there exists $k\geq 1$ such that $s(M_n)_w\subseteq (M_k)_w$ for any $n\geq k$. Hence $sM\subseteq (M_k)_w$, and so $M$ is $S$-$w$-finite $($resp., $S$-$w$-principal$)$ with respect to $s$. It follows by Theorem \ref{u-countably} that $M$ is $u$-$S$-$w$-Noetherian.	
\end{proof}

\begin{theorem}\label{u-countablys}
	Let $R$ be a ring, $S$ an anti-Archimedean multiplicative subset of $R$ and $M$ an $R$-module. Then $M$ is a $u$-$S$-Noetherian module $($resp., $u$-$S$-PIM$)$ if and only if there is $s\in S$ such that  every countably generated submodule of $M$ is $S$-finite $($resp., $S$-principal$)$ with respect to $s$.
\end{theorem}
\begin{proof} It is similar to the proof Theorem \ref{u-countably}. We exhibit it for completeness.

 The necessity clearly holds. Now we prove the sufficiency. 	Let $N=\langle f_{\alpha}\mid \alpha<\gamma\rangle$ be an $R$-submodule of $M$, where $\gamma$ is a cardinal. Consider the index set $$\Gamma=\{\beta\mid s(\langle f_{\alpha}\mid\alpha\leq\beta\rangle/\langle f_{\alpha}\mid\alpha<\beta\rangle)\not=0\}.$$
	
Let $t\in\bigcap\limits_{n\geq 1}s^nR\bigcap S$.	Then we claim that $t(N/ \langle f_\beta\mid \beta\in\Gamma\rangle)=0$. Indeed, let  $f_{\beta'}\in N$. If $\beta'\in\Gamma$, we have done. Otherwise, $sf_{\beta'}\subseteq \langle f_{\alpha}\mid\alpha<\beta'\rangle$. Suppose $\beta'$ is a limit ordinal, then  $sf_{\beta'}\subseteq \langle f_{\alpha}\mid\alpha<\beta''\rangle$ for some successor ordinal $\beta''$, since $sf_{\beta'}$ is finitely generated. So we only consider the successor ordinal case. Now assume $\beta'$ is a successor ordinal, then $sf_{\beta'}\subseteq \langle f_{\alpha}\mid\alpha\leq \beta'-1\rangle$, and hence every element in $sf_{\beta'}$ is a finite linear combination of $f_\alpha$ with $\alpha\leq \beta'-1$. If all such finite $\alpha$ are in $\Gamma$, we have done. Otherwise, we may assume each $f_\alpha$ with $\alpha\not\in\Gamma$ satisfy $sf_\alpha$  is a finite linear combination of $f_\alpha$ with $\alpha\leq \beta'-2.$ Hence every element in  $s^2f_{\beta'}$ is a finite linear combination of $f_\alpha$ with $\alpha\leq \beta'-2.$ Such iterating steps terminate in finite steps. Indeed,  on contrary assume the iterating steps don't terminate. Denote by $A$ the set of all such $\alpha$. Then $A$ is countable. Denote by $K$ the $R$-module generated by all such  $f_\alpha$ with $\alpha\in A$. So there is a finitely (resp., principally) generated submodule $K'$ such that $s(K/K')=0$ by assumption. Taking $B$ to be the set of $\beta$ such that $f_{\beta}$ appears in the finite generators of $K'$, we have $B$ is finite and  $sf_{\beta'}\subseteq\langle f_\beta\mid \beta\in B\rangle$. So each $\alpha\in A-B$ is not in $\Gamma$, and hence the iterating steps terminate in finite steps. It follows that there exists  an integer $n$ such that every element in $s^nf_{\beta'}$ is a finite linear combination of $f_\alpha$ with $\alpha\in \Gamma,$ and hence $t(N/ \langle f_\beta\mid \beta\in\Gamma\rangle)=0$ .

If $\Gamma$ is finite, then $N$ is $S$-finite with respect to $t$, and so we have done (resp., by assumption). Otherwise assume $\Gamma$ is infinite. Take $\Gamma'=\{\beta_1<\beta_2<\cdots\}\subseteq \Gamma.$ Note that $\langle f_\beta\mid \beta\in\Gamma'\rangle$ is $S$-finite (resp., $S$-principal) with respect to $s$. Since every element in $\langle f_\beta\mid \beta\in\Gamma'\rangle$ is a finite combination of $f_\beta$, there is a finite subset $\Gamma''$	of $\Gamma'$ such that $s(\langle f_\beta\mid \beta\in\Gamma'\rangle/\langle f_\beta\mid \beta\in\Gamma''\rangle)=0$. So there exists an integer $i$ such that no $\beta_j\in \Gamma'$ with $j\geq i$, which is a contradiction. Consequently, $M$ is a -Noetherian  module (resp., $u$-$S$-PIM) with respect to $t$.
\end{proof}

\begin{corollary}  {\bf (Eakin-Nagata-Formanek Theorem for $u$-$S$-Noetherian modules)}
	Let $R$ be a ring and $S$ an anti-Archimedean multiplicative subset of $R$. Let $M$ be an $R$-module.  Then the following statements are equivalent:
	\begin{enumerate}
		\item $M$ is $u$-$S$-Noetherian;
		\item there exists $s\in S$ such that any  ascending chain of submodules of $M$ is stationary with respect to $s$.
	\end{enumerate}
\end{corollary}
\begin{proof}
It is similar with the proof of Theorem \ref{u-s-noe-char}.
\end{proof}

In the finial part of this section, We will consider some ``local'' characterizations of classical rings and modules in terms of the new ones.

Let $\p$ be a prime ideal of $R$. We say an $R$-module $M$  is  \emph{$(u$-$)\p$-$w$-Noetherian} (resp., \emph{$(u$-$)\p$-$w$-principal ideal}) provided that  $M$ is $(u$-$)(R\setminus\p)$-$w$-Noetherian (resp., \emph{$(u$-$)R\setminus\p$-$w$-principal ideal}). The next result gives a local characterization of $w$-Noetherian modules.
\begin{theorem}\label{s-noe-m-loc-char}
	Let $R$ be a ring and $M$ an $R$-module. Then the following statements are equivalent:
	\begin{enumerate}
		\item  $M$ is  $w$-Noetherian;
		\item   $M$ is  $(u$-$)\p$-$w$-Noetherian for any $\p\in \Spec(R)$;
		\item   $M$ is $(u$-$)\m$-$w$-Noetherian  for any $\m\in \Max(R)$;
		\item   $M$ is $(u$-$)\p$-$w$-Noetherian  for any $\p\in w\mbox{-}\Spec(R)$;
		\item   $M$ is $(u$-$)\m$-$w$-Noetherian for any $\m\in w\mbox{-}\Max(R)$.
	\end{enumerate}
\end{theorem}
\begin{proof} 
	
	$(1)\Rightarrow (2)\Rightarrow (3)$  and $(1)\Rightarrow (2)\Rightarrow (4)\Rightarrow (5):$  Trivial.
	
	$(3)\Rightarrow (1):$ Let $N$ be a submodule of $M$. Then for each  $\m\in \Max(R)$, there exists an element $s^{\m}\in R\setminus\m$ and a finitely generated subideal $K^{\m}$ of $N$ such that $s^{\m}(N/K^{\m})$ is $\GV$-torsion. Since $\{s^{\m} \mid \m \in \Max(R)\}$ generated $R$, there exist finite elements $\{s^{\m_1},...,s^{\m_n}\}$ such that $\{s^{\m_1},...,s^{\m_n}\}$ generates $R$. Set $K=\sum\limits_{i=1}^nK^{\m_i}$. Then $K$ is finitely generated and $s^{\m_i}(N/K)$ is $\GV$-torsion for each $i=1,\dots,n$. So	we have $N/K=\langle s^{\m_1},\dots,s^{\m_n}\rangle(N/K)$ is also $\GV$-torsion. So $N$ is $w$-finite, and hence $M$ is a $w$-Noetherian module.
	
	$(5)\Rightarrow (1):$  Similar to 	$(3)\Rightarrow (1)$. We only note that 
	since the ideal generated by  $\{s^{\m} \mid \m \in \Max(R)\}$  is not contained in any maximal $w$-ideal of $R$, $\langle s^{\m} \mid \m \in \Max(R)\rangle_w=R$. So one can easily verify that $\langle s^{\m_1},\dots,s^{\m_n}\rangle(N/K)$ is $\GV$-torsion implies that $N/K=\langle s^{\m_1},\dots,s^{\m_n}\rangle_w(N/K)$ is also $\GV$-torsion in the proof of $(3)\Rightarrow (1)$. 
\end{proof}

Recall that a ring $R$ is called a  \emph{generalized  ZPI-ring} if every proper ideal of $R$ is a product of prime ideals. It is well-known that a ring $R$ is a generalized  ZPI-ring if and only if it is a finite direct product of Dedekind domains and PIRs.
It was proved in  \cite[Corollary 13]{ad02} that a ring  $R$ is a generalized  ZPI-ring if and only if  $R$ is a $\m$-PIR for every $\m\in \Max(R)$, if and only if  $R$ is a $\p$-PIR for every $\p\in \Spec(R)$. Now we consider the $w$-version of this result.  

Recall from \cite{J23} that a ring $R$ is called a  \emph{generalized $w$-ZPI}-ring if every proper $w$-ideal of $R$ is a $w$-product of prime $w$-ideals. It was proved in \cite[Theorem 2.3]{J23}  that a ring $R$ is a generalized  $w$-ZPI-ring if and only if $R$ is a $w$-Noetherian ring and $R_\p$ is a $w$-PIR for any maximal (or prime) $w$-ideal $\p$ of $R$,  if and only if it is a finite direct product of Krull domains and PIRs.

\begin{corollary}\label{s-wzpi-m-loc-char}
	Let $R$ be a ring. Then the following statements are equivalent:
	\begin{enumerate}
		\item  $R$ is a generalized $w$-ZPI-ring;
		\item   $R$ is a $\p$-$w$-PIR for any $\p\in w\mbox{-}\Spec(R)$;
		\item   $R$ is an $\m$-$w$-PIR for any $\m\in w\mbox{-}\Max(R)$.
	\end{enumerate}
\end{corollary}
\begin{proof}

	$(1)\Rightarrow (2)$: Suppose  $R$ is a generalized $w$-ZPI-ring. Let $\p$ be a prime $w$-ideal of $R$. Then  $R_\p$ is a $w$-PIR. Since  $R$ is a finite direct product of Krull domains and PIRs. We may assume $R$ is a  Krull domain. Let $I$ be a $w$-ideal of $R$, then $I$ is finitely generated and $I_\p=\langle \frac{r}{s}\rangle$ as $R_\p$-module, and so $tI\subseteq \langle r\rangle\subseteq I$ for some $t\in R-\p$. Hence $R$ is a $\p$-$w$-PIR.
	
	$(2)\Rightarrow (3)$: Trivial.
	
	$(3)\Rightarrow (1)$: Since  $R$ is an $\m$-$w$-PIR for any $\m\in w\mbox{-}\Max(R)$, then $R$ is $w$-Noetherian by Theorem \ref{s-noe-m-loc-char}. Note that $R_\m$ is a $w$-PIR for any $\m\in w\mbox{-}\Max(R)$.  So $R$ is a generalized $w$-ZPI-ring by \cite[Theorem 2.3]{J23}.
\end{proof}

We can use the uniform version of Corollary \ref{s-wzpi-m-loc-char} to characterize $w$-PIRs.
\begin{corollary}\label{s-wpIR-loc-char}
	Let $R$ be a ring . Then the following statements are equivalent:
	\begin{enumerate}
		\item  $R$ is a $w$-PIR;
		\item   $R$ is a $u$-$\p$-$w$-PIR for any $\p\in w\mbox{-}\Spec(R)$;
		\item   $R$ is a $u$-$\m$-$w$-PIR for any $\m\in w\mbox{-}\Max(R)$.
	\end{enumerate}
\end{corollary}
\begin{proof} $(1)\Rightarrow (2)$ and  $(2)\Rightarrow (3)$: Trivial.
	
	$(3)\Rightarrow (1)$: It follows by Corollary \ref{s-wzpi-m-loc-char} that  $R$ is a generalized $w$-ZPI-ring, which is a finite direct product of Krull domains and PIRs. We may assume that $R$ is a Krull domain by Proposition \ref{mutl-usm}. If $R$ is a $w$-PID (i.e., UFD), we have done. Otherwise, there is a non-principal maximal $w$-ideal $\m$.  Let $s\in R-\m$ such that every ideal of $R$ is $S$-$w$-principal with respect to $s$. It follows by Proposition \ref{s-loc-u-noe} that  $R_s$ is a $w$-PID which is a contradiction since $\m_s$ is not a principal.	
\end{proof} 

Concentrating to the $\DW$ ring case, we have the following characterizations of PIRs.

\begin{corollary}\label{s-wpIR-loc-char}
Let $R$ be a ring. Then the following statements are equivalent:
	\begin{enumerate}
		\item  $R$ is a PIR;
		\item   $R$ is a $u$-$\p$-PIR for any $\p\in \Spec(R)$;
		\item   $R$ is a $u$-$\m$-PIR for any $\m\in \Max(R)$.
	\end{enumerate}
\end{corollary}

\begin{acknowledgement}
	
	The author would like to thank the reviewers for many valuable suggestions on revision of this paper. 
\end{acknowledgement}


\begin{thebibliography}{99}



\bibitem{ad02}  D. D.  Anderson, T.  Dumitrescu,  $S$-Noetherian rings, {\it Comm. Algebra} {\bf 30} (2002) 4407-4416.



\bibitem{bh18} D. Bennis,  M. El Hajoui,   On $S$-coherence, {\it J. Korean Math. Soc.} \textbf{55}(6) (2018)  1499-1512.

\bibitem{qkwcz22} M. Z. Chen, H. Kim,  W. Qi, F. G. Wang, W. Zhao, Uniformly $S$-Noetherian rings,  {\it Quaest. Math.}  {\bf47}(5) (2024)
1019-1038.

\bibitem{A18}   A. Hamed,  On $S$-Mori domains, {\it J. Algebra Appl.} {\bf17}(9) (2018) ID 1850171, 11 p.
\bibitem{hh15}  A. Hamed, H. Sana,   $S$-Noetherian rings of the forms $A[X]$ and $A[[X]]$, {\it Comm. Algebra} {\bf43} (2015) 3848-3856.

\bibitem{as16} A. Hamed, H. Sana,  Modules satisfying the $S$-Noetherian property and $S$-ACCR, {\it Comm. Algebra} {\bf 44} (2016) 1941-1951.


\bibitem{J23} J. R. Juett, General $w$-ZPI-rings and a tool for characterizing certain classes of monoid rings,
{\it Comm. Algebra} \textbf{51}(3) (2023) 1117-1134.



\bibitem{lO14} J. W. Lim, D. Y. Oh,  $S$-Noetherian properties on amalgamated algebras along an ideal, {\it J. Pure Appl. Algebra} {\bf 218} (2014) 2099-2123.

\bibitem{lO15} J. W. Lim, D. Y. Oh,  $S$-Noetherian properties on composite ring extensions, {\it Comm. Algebra} {\bf 43}  (2015)  2820-2829.

\bibitem{l07} Z. K.  Liu,  On $S$-Noetherian rings, {\it Arch. Math.} (Brno) {\bf 43} (2007) 55-60.


\bibitem{kkl14} H. Kim,  M. O. Kim,  J. W.  Lim,    On $S$-strong Mori domains, {\it J. Algebra}  {\bf 416}  (2014)  314-332.

\bibitem{kmz21} H. Kim, N. Mahdou and Y. Zahir,  $S$-Noetherian in Bi-amalgamations, {\it Bull. Korean Math. Soc.}   {\bf58}(4) (2021) 1021-1029.




\bibitem{fm97} F. G. Wang,  R. L. McCasland,   On $w$-modules over strong Mori domains, {\it Comm. Algebra} {\bf 25}(4) (1997) 1285-1306.

\bibitem{WK15} F. G. Wang,  H. Kim,    Two generalizations of projective modules and their applications, {\it J. Pure Appl. Algebra}  {\bf 219}(6) (2015) 2099-2123.

\bibitem{fk16} F. G.  Wang,  H. Kim,  {\it  Foundations of Commutative Rings and Their Modules} (Singapore, Springer, 2016).
\bibitem{WQ15} F. G. Wang,   L. Qiao,   The $w$-weak global dimension of commutative rings, {\it  Bull. Korean Math. Soc.}  {\bf 52}(4) (2015) 1327-1338.
\bibitem{x22} S. Q. Xing,  Gorenstein FP$_{\infty}$-injective modules and $w$-Noetherian rings, {\it  Algebra Colloq.} {\bf 29}(4) (2022) 687-712.

\bibitem{xw15} S. Q. Xing, F. G. Wang, A note on $w$-Noetherian rings,  {\it Bull. Korean Math. Soc.} {\bf 52}(2) (2015) 541-548.

\bibitem{yc12} H. Yin, Y. Chen, $w$-overrings of $w$-Noetherian rings, {\it Studia Sci. Math. Hungar.} {\bf 49}(2) (2012) 200-205.

\bibitem{hfxc11}  H. Y. Yin, F. G. Wang,  X. S. Zhu, Y. H. Chen, $w$-modules over commutative rings, {\it J. Korean Math. Soc.} {\bf 48}(1) (2011) 207-222.

\bibitem{zwk} J. Zhang, F. G. Wang, H. Kim,   Injective modules over $w$-Noetherian rings II, {\it J. Korean Math. Soc.} {\bf 50}(6) (2013) 1051-1066.

\bibitem{z21} X. L. Zhang,    Characterizing $S$-flat modules and $S$-von Neumann regular rings by uniformity, {\it Bull. Korean Math. Soc.} {\bf 59}(3) (2022) 643-657.

\bibitem{z23} X. L. Zhang,    On uniformly $S$-absolutely pure modules, {\it J. Korean Math. Soc.} {\bf 60}(3) (2023) 521-536 .

\bibitem{zkq23} X. L. Zhang, H. Kim, W. Qi,  On two versions of Cohen's theorem for modules, {\it Kyungpook Math. J.} {\bf 63}(1)  (2023) 29-36. 



\bibitem{zkzx22} D. C. Zhou, H. Kim, X. L. Zhang, J. Xie, Some characterizations of $w$-Noetherian rings and SM rings, {\it J. Math.} {\bf 4} (2022) ID 7403502, 11 p.


\end{thebibliography}
\end{document}